\documentclass[reqno, 12pt]{amsart}

\usepackage{srcltx}
\usepackage[T2A]{fontenc}
\usepackage[utf8]{inputenc}
\usepackage[russian,english]{babel}
\usepackage{a4wide}

\usepackage{amsfonts, amsthm, amsmath, amssymb}
\numberwithin{equation}{section} 
\newtheorem{theorem}{Theorem}[section]
\newtheorem{lemma}[theorem]{Lemma}
\newtheorem{cor}[theorem]{Corollary}

\newtheorem{proposition}[theorem]{Proposition} 

\newtheorem{conjecture}[theorem]{Conjecture}

\theoremstyle{definition}

\renewcommand{\phi}{\varphi}

\renewcommand{\L}{\mathcal L}

 \newcommand{\E}{\mathcal E}

\renewcommand{\leq}{\leqslant}

\renewcommand{\geq}{\geqslant}

\renewcommand{\d}{{\rm d}}
 \newcommand{\e}{{\rm e}}

\renewcommand{\rho}{\varrho}

\newcommand{\eps}{\varepsilon}

\theoremstyle{plain}
\newtheorem{X}{X}[section] 
\theoremstyle{remark}
\newtheorem{remark}[X]{Remark}

\begin{document}
\title[
]{On a conjecture of Montgomery and Soundararajan}

\date{\today}
 
\author{R.\ de la Bret\`eche}
\address{
Institut de Math\'ematiques de Jussieu-Paris Rive Gauche\\
Universit\'e de Paris, Sorbonne Universit\'e, CNRS UMR 7586\\
Case Postale 7012\\
F-75251 Paris CEDEX 13\\ France}
\email{regis.delabreteche@imj-prg.fr}

\author{D. Fiorilli}
\address{CNRS, Universit\'e Paris-Saclay, Laboratoire de math\'ematiques d'Orsay, 91405, Orsay, France.}
\email{daniel.fiorilli@universite-paris-saclay.fr}

\maketitle

\begin{abstract}
We establish lower bounds for all weighted even moments of primes up to $X$ in intervals which are in agreement with a conjecture of Montgomery and Soundararajan. Our bounds hold unconditionally for an unbounded set of values of $X$, and hold for all $X$ under the Riemann Hypothesis. We also deduce new unconditional $\Omega$-results for the classical prime counting function. 
\end{abstract}

\section{Introduction}

The goal of this paper is to investigate~\cite[Conjecture 1]{MS04}.
Let 
\begin{equation}
    \mu_n:= \begin{cases}
    \frac{(2m)!}{2^m m!} & \text{ if } n=2m \text{ for some } m\in \mathbb N,\\
    0 &\text{otherwise}
    \end{cases}
\end{equation}
be the $n$-th moment of the Gaussian.

\begin{conjecture}[Montgomery, Soundararajan]
\label{conjecture montgomery sound}
Fix $\eps>0$. For each fixed $n\in  \mathbb N$ and uniformly for $ \frac{(\log X)^{1+\eps}}{X}\leq \delta \leq \frac{1}{X^{\eps}}$,
\begin{equation}
     \frac {1}{ X}\int_1^{X} \frac{(\psi(x+\delta X)-\psi(x)-\delta X)^n}{X^{\frac n2}} \d x   = (\mu_n+o(1)) \big( \delta  \log(\delta^{-1})\big)^{\frac n2}. 
     \label{equation MS}
\end{equation}
\end{conjecture}
In the range $ X^{-1} (\log X)^{1+\eps} \leq \delta \leq X^{-1+\frac 1n} $, Montgomery and Soundararajan~\cite[Theorem 3]{MS04} have shown that~\eqref{equation MS} follows from a strong form of the Hardy-Littlewood prime $k$-tuple conjecture. 
They also mention that the conjecture could also hold whenever $\delta=o(1)$. For applications on the distribution of gaps between primes, see for instance~\cite{FGL18}. 

Currently, many results towards Conjecture~\ref{conjecture montgomery sound} are known in the case $n=2$ (see the remarks following Theorem~\ref{theorem main lower bounds moments} below for a description of the work of Selberg, Goldston, Montgomery, and others on this topic), but little is known for higher moments. This is in contrast with the theory of moments of $L$-functions, in which we have lower and upper bounds of the correct order of magnitude for higher moments in several different families thanks to the work of Ramachandra~\cite{R80}, Rudnick and Soundararajan~\cite{RS05}, Soundararajan~\cite{S09}, Harper~\cite{H13}, Radziwi\l\l-Soundararajan~\cite{RS15}, and others.

In the current paper, we establish lower bounds for a weighted version of~\eqref{equation MS} for all even $n$, for values of $\delta$ that are relatively close to $1$. In addition to being the first estimate on higher moments, we believe that our bounds are sharp up to a power-saving error term in $\delta$ (c.f.~\cite[Theorem 3]{MS04}).
Prior to our work, the order of magnitude of the left hand side of~\eqref{equation MS} and some variants was known under RH for $n=2$ and in various ranges of~$\delta$. However, the determination of the exact asymptotic size has been shown to be strongly related with deep simplicity and pair-correlation type estimates~\cite{GM78,M83,GM87,G88,MS02,LPZ12,BKS16}. 

The key idea which will allow us to circumvent the need to understand spacing statistics 
and Diophantine properties (for higher moments) of zeros of the zeta function is a 
 positivity argument in the explicit formula. Such an argument in conjunction with Parseval's identity has been successfully used in previous works on the variance (see e.g.~\cite{GG90}), however the novelty in the present paper is to avoid the need for Parseval's identity (in particular for higher moments). 

For any fixed $\kappa>0$, we define the class of test functions  $\E_\kappa\subset \mathcal L^1(\mathbb R)$ to be the set of all differentiable\footnote{One can replace differentiability by a Lipschitz condition if for instance $\eta$ is compactly supported in $\mathbb R$ and monotonous on $\mathbb R_{\geq 0}$.
} even $ \eta :\mathbb{R} \rightarrow \mathbb{R}$ such that for all $t\in \mathbb{R}$, 
\begin{equation}\label{majeta}
    \eta (t) ,  \eta'(t) \ll \e^{-\kappa |t|}, 
\end{equation}
moreover $\widehat \eta(0)>0$ and for all  $\xi \in  \mathbb{R}$ we have that\footnote{We can take for example $\eta = \eta_0 \star \eta_0$ for some smooth and rapidly decaying $\eta_0$.}
 \begin{equation}
 \label{condiavecdelta}
 0 \leq \widehat \eta(\xi) \ll (|\xi|+1)^{-1} \log(|\xi|+2)^{-2-\kappa}. 
 \end{equation}
We consider the following weighted version of $x^{-\frac 12}(\psi(x+\delta x)-\psi(x)-\delta x)$. For $\eta\in \E_{\kappa}$ and $\delta < 2\kappa$, we define
$$ \psi_{\eta} (x,\delta):= \sum_{n\geq 1} \frac{\Lambda(n)}{n^{\frac 12}} \eta\Big(\delta^{-1}\log \Big(\frac nx\Big)\Big). $$
Morally, this function counts prime powers in the interval $[x (1-O(\delta)),x(1+O(\delta))]$, in which the weight $n^{-\frac 12}$ is equal to $x^{-\frac 12}(1+O(\delta))$.
The expected main term for  $\psi_{\eta} (x,\delta)$ is given by
$$ \int_{0}^{\infty} \frac{\eta(\delta^{-1}\log (\frac tx))}{t^{\frac 12}} \d t = x^{\frac 12} \delta \int_{\mathbb R} \e^{\frac{\delta w}2} \eta(w) \d w 
, $$
which we will denote by $x^{\frac 12} \delta \L_\eta(\frac \delta2)$ (note that for $\delta<\kappa$, $\L_\eta(\frac \delta 2)=\L_\eta(-\frac \delta 2)= \widehat \eta(0)+O(\delta) $). Subtracting this main term is equivalent to summing $\Lambda(n)-1$ instead of $\Lambda(n)$ (more precisely, it is equivalent to working with the measure $\d(\psi(t)-t)$). We also consider the set $\mathcal U$ of non-trivial even integrable functions $\Phi:\mathbb R \rightarrow \mathbb R$ such that $\Phi,\widehat \Phi\geq 0$ (in particular, $\Phi(0)>0$).
Finally, for $h:\mathbb R \rightarrow \mathbb R$ we define 
\begin{equation}
\alpha(h):=  \int_{\mathbb R } h(t)\d t; \qquad \beta(h):=  \int_{\mathbb R } h(t)  (\log | t|)\d t,
\end{equation}
whenever these integrals converge.
Here is our main RH result on the $n$-th moment
$$ M_n(X,\delta;\eta,\Phi) :=\frac {1}{ (\log X) \int_0^{\infty }\Phi}\int_{1}^{\infty} \Phi\Big( \frac {\log x}{\log X}\Big)\big(\psi_\eta(x,\delta)-x^{\frac 12}\delta \L_\eta(\tfrac \delta 2)\big)^{n} \frac{\d x}x.$$

\begin{theorem}
\label{theorem main lower bounds moments}
Assume RH, and let $0<\kappa < \frac 12$, $\eta\in \E_{\kappa}$, $\Phi\in \mathcal U$. For $n\in \mathbb N$, $X\in \mathbb R_{\geq 2}$, $\delta \in (0,\kappa)$, and in the range $n\leq \delta^{-\frac 12}(\log(\delta^{-1}+2))^{\frac 12}$, we have that
\begin{equation}
\label{equation nth moment first bound}
\begin{split}
(-1)^n M_n(X,\delta;\eta,\Phi)  \geq  
 \mu_n  \delta^{\frac n2}\big(\alpha(\widehat \eta^2)  \log (\delta^{-1})+\beta(\widehat \eta^2)&\big)^{\frac n2} \Big(1+O_{\kappa,\eta}\Big(\frac{ n^2 \delta}{  \log(\delta^{-1}+2)}\Big)\Big) \\ & +O_{\Phi}\Big(\delta\frac{(K_{\eta}  \log(\delta^{-1}+2))^{n 
}}{\log X}\Big),
\end{split}
\end{equation}
where the implied constants and $K_\eta>0$ are independent of $n, X$ and $\delta$.
\end{theorem}

\begin{remark}
\begin{enumerate}
    \item For $n=2$ and in the range $ X^{-c({\eta,\Phi})} \leq \delta \leq 1 $,~\eqref{equation nth moment first bound} implies a lower bound with the predicted main term as well as a secondary term conjectured in the work of Montgomery and  Soundararajan~\cite[(2)]{MS02}. Here, $c({\eta,\Phi})>0$ is a constant. Variants of this particular case (with various weights and  measures) have attracted a lot of attention since Selberg's foundational work~\cite{S43}. This includes Goldston and Montgomery's RH upper bound~\cite{GM87} in the whole range $0<\delta \leq 1$, Saffari and Vaughan's unconditional upper bound~\cite{SV77} in the range $x^{-\frac 56+\eps}\leq \delta \leq 1 $, Goldston's GRH lower bound~\cite{G84,G95} in the range $x^{-1}\leq \delta \leq x^{-\frac 34} $ (unconditional for $x^{-1}\leq \delta \leq x^{-1} (\log x)^A $), its generalization by Özlük~\cite{O87} and Goldston and Yildirim~\cite{GY98,GY01} to a fixed arithmetic progression, Zaccagnini's unconditional upper bound~\cite{Z96,Z98} in the range $x^{-\frac 56-\eps} \leq \delta \leq 1 $ (building on the work of Huxley~\cite{H72} and Heath-Brown~\cite{HB88}),
    and others.

    \item For $n=2m$ with $m\geq 2$ and in the interval $ (\log X)^{-\frac 1{m-1} +o(1)}\leq  \delta \ll 1$, we obtain a lower bound which is in agreement with Conjecture~\ref{conjecture montgomery sound}.
    
    \item Goldston and Yildirim~\cite{GY03,GY07} have computed the first three moments of a related quantity involving a major arcs approximation of $\Lambda(n)$, and deduced that in the range $X\leq x \leq 2X$, $ X^{-1} (\log X)^{14} \ll \delta \leq X^{-\frac 67-\eps}$ and under GRH, 
    $\psi(x+\delta X)- \psi(x) - \delta X =\Omega_{\pm}((\delta x \log x)^{\frac 12})$.
    
    \item In the function field case, estimates for the variance of $\Lambda(n)$ and more general arithmetic sequences have been established by Keating and Rudnick~\cite{KR14,KR16} and Rodgers~\cite{R18}. Moreover, Hast and Matei~\cite{HM19} have given a geometric interpretation for the higher moments.
\end{enumerate}
\end{remark}

We now rephrase Theorem~\ref{theorem main lower bounds moments} and state our unconditional results.
\begin{cor}
\label{corollary}
Let $0<\kappa<\frac 12$, $\eta\in \E_{\kappa}$, and $\Phi\in \mathcal U$.
Let moreover $f : \mathbb R_{\geq 0} \rightarrow (0,\frac 12]$ be any function such that $\lim_{x \rightarrow \infty} f(x) = 0$, and let $\delta \in (0,1)$, $m \in \mathbb N$ and $X \in \mathbb R_{\geq 2}$
be such that either $m=1$ and $\delta \in (X^{-f(X)},f(X)]$, or  $2\leq m\leq \min(\delta^{-\frac 12} (\log(\delta^{-1}+2))^{\frac 12}f(X)^{\frac 12},\log\log X)$ and
$\delta \in ((\log X)^{-\frac 1{m-1}} (\log\log X)^4,f(X)].$ Then under RH  
we have that
\begin{equation}
\label{equation nth moment first bound corollary} 
 M_{2m}(X,\delta;\eta,\Phi) \geq  
 \mu_{2m}  \delta^{m}\big(\alpha(\widehat \eta^2)  \log (\delta^{-1})+\beta(\widehat \eta^2)\big)^{m}\Big(1+O\Big(f(X)+ \frac 1{(\log (\delta^{-1}))^2}\Big)\Big).
\end{equation}
Unconditionally, there exists a sequence $\{X_j\}_{j\geq 1}$ tending to infinity such that whenever $X=X_j$,~\eqref{equation nth moment first bound corollary} holds with
$m=1$ and $\delta \in (X^{-f(X)},f(X)]$. The same statement holds in the range $2\leq m\leq \min(\delta^{-\frac 12},\log\log X)$ and $\delta \in ((\log X)^{-\frac 1{ m-1} }(\log\log X)^4,f(X)]$.
\end{cor}
We now state our unconditional $\Omega$-results for the usual prime counting function in short intervals $\psi(x+\delta x)-\psi(x) -\delta x$. Note that this quantity has standard deviation of order $(\delta x \log(\delta^{-1}+2))^{\frac 12} $. We will show that $\psi(x+\delta x)-\psi(x) -\delta x$ can be larger than an unbounded constant times this.

  \begin{cor}
\label{corollary very large values}
Let $\eps>0$ be small enough. There exists a sequence $\{(x_j,\delta_j)\}_{j\geq 1}$ with $\delta_j\in \big[\eps \frac{(\log_3 x_j)^{\frac 92}}{(\log x_j)^{2} (\log_2 x_j)^{\frac 52}},2\frac{(\log_3 x_j)^3}{(\log_2 x_j)^{2}}\big]$, $\lim_{j\rightarrow \infty} x_j=\infty$, and such that 
$$ \big|\psi(x_j+\delta_j x_j)-\psi(x_j) - \delta_j x_j \big| \gg  \delta_j^{-\frac 14}(\log(\delta_j^{-1}+2))^{\frac 14} \cdot \big(\delta_j x_j \log(\delta_j^{-1}+2)\big)^{\frac 12}. $$
If instead we require that $\delta_j\in \big[(\log x_j)^{-\frac 72-\frac 3{2M}},(\log x_j)^{-\frac 1{M+1}}\big]$ for some large fixed $M\in \mathbb Z_{\geq 2}$, then we can choose the sequence $\{(x_j,\delta_j)\}_{j\geq 1}$ in such a way that
$$ \big|\psi(x_j+\delta_j x_j)-\psi(x_j) - \delta_j x_j \big| \gg   M^{\frac 12}\cdot\big(\delta_j x_j \log(\delta_j^{-1}+2)\big)^{\frac 12}. $$

\end{cor}

\section{Proof of Theorem~\ref{theorem main lower bounds moments}}
Throughout this section, we will denote by $\rho=\beta+i\gamma$ the non-trivial zeros of the Riemann zeta function. We recall the Riemann-von Mangoldt formula
\begin{equation}
    N(T):=\{  \rho : 0 \leq \Im m(\rho) \leq T \}  = \frac{T}{2\pi} \log \frac{T}{2\pi e} +O(\log (T+2)),
    \label{equation Riemann von Mangoldt}
\end{equation}
which is valid for $T \geq 0$.

A major ingredient in our proof is the following explicit formula for $\psi_\eta  (x,\delta)$ and a related quantity. 
\begin{lemma}
\label{lemma explicit formula}
Let $0<\kappa < \frac 12$ and $\eta \in\E_\kappa$. For $t\geq 0$ and $0<\delta < \kappa $ we have the formulas
\begin{equation} 
\label{equation explicit formula psi}
    \psi_\eta (\e^t,\delta) -
    \e^{\frac t2 } \delta \L_\eta(\tfrac \delta2)=- \delta\sum_{\rho} \e^{   (\rho-\frac 12) t} \widehat \eta\Big(\frac{\delta}{2\pi}\frac{\rho-\frac 12}i\Big)+O_{\kappa,\eta}(
   E_{\kappa,\eta}(t,\delta));
\end{equation} 
\begin{equation} 
\label{equation explicit formula psi sans racine}
\begin{split} 
  \e^{-\frac t2}\Big( \sum_{n\geq 1} \Lambda(n) \eta( \delta^{-1} ( \log n - t  ) ) -
    \e^{t } \delta \L_\eta(\delta)\Big)=- \delta\sum_{\rho} &\e^{   (\rho-\frac 12) t} \widehat \eta\Big(\frac{\delta}{2\pi}\frac{\rho}i\Big)\cr &+O_{\kappa,\eta}\big(
  \e^{-\frac t2}(\delta+ E_{\kappa,\eta}(t,\delta))\big),
\end{split}
\end{equation}
where $\rho$
runs over the nontrivial zeros of $\zeta(s)$, and 
\begin{equation}
     E_{\kappa,\eta}(t,\delta):=  \begin{cases}
     \delta\e^{-  \frac {  t} 2  }+\log (\delta^{-1}+2)\e^{-\frac{\kappa  t}{\delta}  }  & \text{ if } t\geq 1 ,\\
\frac{\delta} t+\log (\delta^{-1}+2)\e^{-\frac{\kappa  t}{\delta}  }  & \text{ if } \delta \leq t <1 ,\\
\log(\delta^{-1}+2) & \text{ if } 0\leq t \leq \delta.
 \end{cases}
 \label{equation bound on E}
\end{equation}  
Under RH we have the uniform bound
\begin{equation}
    \psi_\eta  (\e^t,\delta) -
    \e^{ \frac t2  } \delta \L_\eta( \tfrac \delta 2) \ll_{\kappa,\eta}  \log(\delta^{-1}+2).
    \label{equation uniform bound}
\end{equation}
 
    If in addition to RH we assume that  $\widehat \eta(s) \ll (1+|s|)^{-2-\eps}$ for some $\eps\geq 0$ and whenever $|\Im m(s)|\leq \frac 12$, then we have the estimate
\begin{equation}
\label{equation explicit formula psi approximation}
\begin{split}
  \e^{-\frac t2}\Big( \sum_{n\geq 1} \Lambda(n) \eta( \delta^{-1} ( \log n - t  ) )& -
    \e^{t } \delta \L_\eta(\delta)\Big) =   \psi_\eta (\e^t,\delta) -
    \e^{\frac t2 } \delta \L_\eta(\tfrac \delta2) \cr &+
   O_{\kappa,\eta}\big(\delta^{\frac 12+ \frac{\eps}{2(2+\eps)} } \log(\delta^{-1}+2)+ E_{\kappa,\eta}(t,\delta)\big).
\end{split}
\end{equation}

\end{lemma}
\begin{proof}
To show~\eqref{equation explicit formula psi} we apply \cite[Theorem 12.13]{MV07} with 
$F(u):=\eta ( \frac{t+2\pi u}{\delta})$, so that $\widehat F(\xi) = \e^{i \xi t}\frac{\delta}{2\pi}\widehat \eta(\frac{\delta\xi}{2\pi}) $. 
We obtain that 
\begin{align*}
    \psi_{\eta} ( \e^t,\delta) & -
    \e^{\frac t2  } \delta \L_\eta(\tfrac \delta2)+ \delta\sum_{\rho} \e^{   (\rho-\frac 12)t} \widehat \eta\Big(\frac{\delta}{2\pi}\frac{\rho-\frac 12}i\Big)
    \cr  = &\e^{- \frac t2  }{\delta}
 \int_{\mathbb R} \e^{\frac{\delta  x} 2  }\eta(x )\d x-
  \sum_{n\geq 1} \frac{\Lambda(n)}{n^{\frac 12}} \eta\Big( \frac{t+\log n}{\delta}\Big)
 +\Big(  \frac{\Gamma’}{\Gamma}\Big(\frac 14\Big)-\log \pi \Big)\eta \Big( \frac{t }{\delta}\Big) 
 \cr& +\int_0^\infty \frac{\e^{-\frac x2}}{1-\e^{-2x}}\Big\{ 2\eta \Big( \frac{t }{\delta}\Big)-\eta \Big( \frac{t+  x}{\delta}\Big)-\eta \Big( \frac{t-  x}{\delta}\Big)\Big\}\d x. 
\end{align*}
A careful analysis of the second integral yields the bound~\eqref{equation bound on E} whenever $\eta\in \E_\kappa$.

The proof of~\eqref{equation explicit formula psi sans racine} is similar, with the choice 
$F(u):=\e^{ - \pi u }\eta ( \frac{t+2\pi u}{\delta})$, so that 
$$ \int_{\mathbb R}F(u) \e^{-(\xi-\frac 12)2\pi u} \d u =  \e^{ \xi t}\frac{\delta}{2\pi}\widehat \eta\Big(\frac{\delta\xi}{2\pi i }\Big).$$

The uniform bound~\eqref{equation uniform bound} follows from the triangle inequality and a straightforward application of the Riemann-von Mangoldt formula~\eqref{equation Riemann von Mangoldt}.

We now move to~\eqref{equation explicit formula psi approximation}. It is sufficient to establish the bound
$$ \delta\sum_{\rho} \e^{   \rho t} \widehat \eta\Big(\frac{\delta}{2\pi}\frac{\rho}i\Big)-\delta\sum_{\rho} \e^{   \rho t} \widehat \eta\Big(\frac{\delta}{2\pi}\frac{\rho-\frac 12}i\Big) \ll_{\kappa,\eta} \e^{\frac t2}\delta^{\frac 12+ \frac{\eps}{2(2+\eps)} } \log(\delta^{-1}+2).$$
To show this, we first truncate the infinite sums. Our conditions on $\eta$ imply that
$$  \delta\sum_{|\rho|> \delta^{- \frac{3+\eps}{2+\eps}}} \e^{   \rho t} \widehat \eta\Big(\frac{\delta}{2\pi}\frac{\rho}i\Big)-\delta\sum_{|\rho|> \delta^{-\frac{3+\eps}{2+\eps}}} \e^{   \rho t} \widehat \eta\Big(\frac{\delta}{2\pi}\frac{\rho-\frac 12}i\Big) \ll_{\kappa,\eta} \e^{\frac t2}\delta^{\frac 12+ \frac{\eps}{2(2+\eps)} } \log(\delta^{-1}+2). $$
The rest of the sums over $\rho$ is bounded by combining~\eqref{equation Riemann von Mangoldt} with the bound
\begin{equation}\begin{split}
   \widehat \eta\Big(\frac{\delta }{2\pi}\frac{\rho-\frac 12}i\Big) - \widehat \eta\Big(\frac{\delta }{2\pi} \frac \rho i\Big)&=\int_{\mathbb R} (\e^{ \delta  (\rho-\frac 12) \xi} - \e^{  \delta \rho  \xi}) \eta (\xi) \d \xi \\ &
     \ll  \int_{|\xi|\leq \delta^{-1} } \delta |\xi \eta(\xi)| d\xi+ \int_{|\xi| > \delta^{-1} } \e^{\frac{\delta |\xi|}{2}} | \eta(\xi)| d\xi \\ &\ll_{\kappa,\eta} \delta+\frac{\e^{-\delta^{-1} \kappa}}{ \kappa-\frac \delta 2} \ll_\kappa \delta.
\end{split}\label{equation bound difference}\end{equation}
\end{proof}

The following estimate on a convergent sum over zeros will be helpful in calculating the main terms in our lower bounds on moments.
 
\begin{lemma}
\label{lemma convergent sum}
Let $0<\kappa <\frac 12$, and let $ h:\mathbb R\rightarrow \mathbb R$ be a measurable function
such that for all $\xi \in \mathbb R$, $ 0\leq h(\xi)\ll (|\xi|+1)^{-2}(\log(|\xi|+2))^{-2-\kappa}$,
and\footnote{The integrability of $\xi h(\xi)$ implies that $\widehat h$ is differentiable (see \cite[p. 430]{KF89}).} for all $t\in \mathbb R$, $ \widehat h(t),\widehat h'(t) \ll  \e^{-\kappa |t|} $.
For $0<\delta <2\kappa $ we have that 

\begin{equation}
    \sum_{\substack{\rho}} h\Big(\frac{\delta }{2\pi}\frac{\rho-\frac 12}i\Big)= \alpha(h)\delta^{-1} \log(\delta^{-1})+\beta(h) \delta^{-1}+O_{\kappa,h}(1),
    \label{equation sum over rho h}
\end{equation} 
where $\rho$ is running over the non-trivial zeros of the Riemann zeta function, and where $h$ is extended to $\{s\in \mathbb C:  |\Im m(s)| < \frac {\kappa}{2\pi } \}$ by writing
\begin{equation}
     h(z):= \int_{\mathbb R} \e^{2\pi i z \xi} \widehat h(\xi) \d \xi.  
     \label{equation h Fourier}
\end{equation}
\end{lemma}

\begin{proof}
The claimed estimate can be established with a slightly weaker error term (and a different class of functions $h$) using the Riemann-von Mangoldt formula~\eqref{equation Riemann von Mangoldt} and the bound 
\begin{equation}
     h\Big(\frac{\delta }{2\pi}\frac{\rho-\frac 12}i\Big) - h\Big(\frac{\delta\Im m(\rho) }{2\pi}\Big)\ll_{\kappa,h} \delta,
      \label{equation approximation rho gamma}
\end{equation} 
which follows from a calculation similar to~\eqref{equation bound difference}.
To obtain the claimed error term, we will use a different technique. Applying the explicit formula~\cite[Theorem 12.13]{MV07} with $F(x):=2\pi\delta^{-1} \widehat h(-2\pi \delta^{-1} x), $ we obtain that
\begin{equation}\label{bsumbj}\sum_{\substack{\rho}} h\Big(\frac{\delta }{2\pi}\frac{\rho-\frac 12}i\Big)=\delta^{-1}\big(b_1(h)+b_2(h)+I(h)\big)+h\Big(\frac{i\delta}{4\pi}\Big)+h\Big(-\frac{i\delta}{4\pi}\Big),\end{equation}
where 
$$ b_1(h):=
\Big(\frac{\Gamma’}{\Gamma}\Big(\frac 14\Big)-\log \pi\Big)\widehat h(0);\qquad
b_2(h):=-\sum_{n=1}^\infty\frac{\Lambda(n)}{ n^{\frac 12}}\big(\widehat h(\delta^{-1}\log n)+\widehat 
h(-\delta^{-1}\log n)\big);$$
$$ 
I(h):=
 \int_0^\infty \frac{\e^{-\frac x2 }}{1-\e^{-2 x}}\big(2\widehat h(0)-\widehat h(\delta^{-1} x)-\widehat h(-\delta^{-1} x)\big)\d x.
 $$ 
Integration by parts shows that
$b_2(h) \ll_\kappa   2^{-\kappa \delta^{-1}}$. We split the integral $I(h)$ into the three ranges $[0,\delta]$, $[ \delta,1]$, $[1,+\infty)$, and denote by $I_1(h),I_2(h),I_3(h)$ the respective integrals. 
We have that 
$$I_3(h) = \widehat h(0)\int_1^\infty \frac{2\e^{-\frac x2 }}{1-\e^{-2 x}}\d x+O_h(\e^{-\frac {\kappa}{\delta}}).
$$ 
Moreover, 
\begin{align*}I_2(h)
=\widehat h(0)\log (\delta^{-1})+ \widehat h(0)\int_0 ^1\!\!\Big( \frac{2\e^{-\frac x2 }}{1-\e^{-2 x}}-\frac 1x\Big)\d x-\!\int_{\mathbb R} h(\xi)\int_1^\infty \!\!\!\cos (2\pi x \xi )\frac{\d x}{x}\d \xi  
+O_h(\delta).
\end{align*}
As for $I_1(h)$, we obtain that
\begin{align*}
   I_1(h)= \int_{\mathbb R} h(\xi)\int_0^1 (1-\cos ({2\pi x\xi} )) \frac{ \d x}{x} \d \xi +O_h(\delta).
\end{align*}
Collecting our estimates for $I_1(h),I_2(h),I_3(h)$ as well as the estimate $h(\pm \frac{i\delta}{4\pi}) = h(0)+O_h(\delta) $, we deduce that 
$$\delta \sum_{\substack{\rho}} h\Big(\frac{\delta }{2\pi}\frac{\rho-\frac 12}i\Big)= \widehat h(0)\big(\log (\delta^{-1})+ C\big)
+\int_{\mathbb R} h(\xi) \log|\xi| \d \xi +O_{\kappa,h}(\delta),
$$
where
\begin{multline*} 
C :=   \int_0 ^1\Big( \frac{2\e^{-\frac x2 }}{1-\e^{-2 x}}-\frac 1x\Big)\d x +\int_1^\infty \frac{2\e^{-\frac x2 }}{1-\e^{-2 x}}\d x  + 
\frac{\Gamma’}{\Gamma}\Big(\frac 14\Big)-\log \pi\\ +  \int_0^1 (1-\cos (2\pi x ))\frac{\d x}{x}-\int_1^\infty \cos (2\pi x )\frac{\d x}{x} .
\end{multline*}
We will show that $C=0$, from which the claimed estimate follows. We have the identity~\cite[\S II.0, Exercise 149]{T15}
 $$
 \frac{\Gamma’}{\Gamma}\Big(\frac 14\Big)=
\int_0^\infty \Big(\frac{ \e^{-2x }}{x}- \frac{2\e^{-\frac x2 }}{1-\e^{-2 x}}\Big) \d x.
$$ 
 We deduce that
 $$C=  \int_0^\infty \frac{ \e^{-2x} -  \cos  (2x)}{x}\d x,$$
which is readily shown to be equal to zero using the residue theorem. 
\end{proof}

\goodbreak
We will also need the following combinatorial lemma.

\begin{lemma}
\label{lemma lower bound 2k sum over zeros}
Let $0<\kappa < \frac 12$, $\eta\in \E_\kappa$, and assume\footnote{One can obtain a slightly weaker but unconditional lower bound by applying~\eqref{equation approximation rho gamma} at the end of the argument.} RH. For $\delta \in (0,\kappa)$, $m\in \mathbb N$, and in the range $m\leq \delta^{-\frac 12}(\log(\delta^{-1}+2))^{\frac 12} $, we have the lower bound
$$\delta^{2m}\!\!\!\!\!\!\!\sum_{\substack{\gamma_1,\dots,\gamma_{2m} \\ \gamma_1+\dots+\gamma_{2m}=0}}\!\!\!\!\!\! \!\!\widehat  \eta\Big(\frac{\delta\gamma_1}{2\pi}\Big)\cdots \widehat\eta\Big(\frac{\delta\gamma_{2m}}{2\pi}\Big)\geq \mu_{2m} \delta^m  \big(\alpha(\widehat \eta^2)  \log (\delta^{-1})+\beta(\widehat \eta^2)\big)^{m}\Big(1+O_{\kappa,\eta}\Big(\frac{ m^2 \delta}{  \log(\delta^{-1}+2)}\Big)\!\Big),  $$
where the $\gamma_j$ are running over the imaginary parts of the non-trivial zeros of the Riemann zeta function.
\end{lemma} 

\goodbreak
\begin{proof}

We will show that
\begin{equation}
 M_{2m}:=\sum_{\substack{\gamma_1,\dots,\gamma_{2m} \\ \gamma_1+\dots+\gamma_{2m}=0}} \widehat  \eta\Big(\frac{\delta\gamma_1}{2\pi}\Big)\cdots \widehat\eta\Big(\frac{\delta\gamma_{2m}}{2\pi}\Big)\geq \mu_{2m}\big( s_2^{m}-m(m-1)s_2^{m-2}s_4\big),  
 \label{equation bound to prove in lemma lower bound 2k sum}
\end{equation} 
 where $s_{2j}:= \sum_{\substack{\gamma}} |\widehat \eta(\frac{\delta\gamma}{2\pi})|^{2j}$. 
Combining this bound with Lemma~\ref{lemma convergent sum} with $h=|\widehat \eta|^2=\widehat \eta^2$ and $h=|\widehat \eta|^4=\widehat \eta^4$ implies the claimed bound. One can check that $\eta\in \mathcal E_\kappa $ implies that for both those choices of $h$, we have the bounds $ \widehat h(t), \widehat h'(t) \ll (|t|^3+1) \e^{-\kappa |t|} $.

Now, to establish~\eqref{equation bound to prove in lemma lower bound 2k sum}, note that this is an equality for $m=1$, and is clear for $m=2$. In the general case, we have that 
$$ M_{2m}\geq  M_{2m}':=\sum_{\substack{\gamma_1,\dots,\gamma_{2m} \text{ distinct} \\ \gamma_1+\dots+\gamma_{2m}=0}} \widehat  \eta\Big(\frac{\delta\gamma_1}{2\pi}\Big)\cdots \widehat\eta\Big(\frac{\delta\gamma_{2m}}{2\pi}\Big).
$$
Note that $M_2=M_2'=s_2$. One can restrict the sum in $M'_{2m}$ to those $2m$-tuples of zeros for which for each $1\leq j\leq 2m$, there exists $1\leq i \leq 2m$, $i\neq j$, such that $\gamma_i=-\gamma_j $. In other words, for each involution $\pi: \{ 1, \dots ,2m\} \rightarrow \{ 1, \dots ,2m\} $ with no fixed points, there exists a subset of $2m$-tuples of zeros $\gamma_1,\dots \gamma_{2m}$ such that for each $1\leq j \leq 2m$, $\gamma_{j} = -\gamma_{\pi(j)}$. Note also that since the $\gamma_j$ are distinct in $M'_{2m}$, the sets of $2m$-tuples associated to different involutions $\pi$ are distinct. Since the total number of such involutions is equal to $\mu_{2m}$, it follows that
$$ M_{2m}'= \mu_{2m} \sum_{\substack{\gamma_1  }}\Big|\widehat \eta\Big(\frac{\delta\gamma_1 }{2\pi}\Big)  
\Big|^2
\sum_{\substack{\gamma_3\notin\{ \gamma_1,-\gamma_1\}  }}\Big|\widehat \eta\Big(\frac{\delta\gamma_3 }{2\pi}\Big)  \Big|^2\ldots \!\!\!\!\!
\sum_{\substack{\gamma_{2m-1}\notin\{ \gamma_1,-\gamma_1,\ldots,  \gamma_{2m-3},-\gamma_{2m-3}\}  }}\Big|\widehat \eta\Big(\frac{\delta\gamma_{2m-1} }{2\pi}\Big)  \Big|^2.  $$
Therefore, by symmetry we have that
\begin{align*}\frac{M_{2m}'}{\mu_{2m}}&=\sum_{\substack{\gamma_1  }}\Big|\widehat \eta\Big(\frac{\delta\gamma_1 }{2\pi}\Big)  
\Big|^2
\sum_{\substack{\gamma_3\notin\{ \gamma_1,-\gamma_1\}  }}\!\Big|\widehat \eta\Big(\frac{\delta\gamma_3 }{2\pi}\Big)  \Big|^2\ldots \Bigg\{ s_2-2\Big|\widehat \eta\Big(\frac{\delta\gamma_1 }{2\pi}\Big)  
\Big|^2\!\!-\ldots-2\Big|\widehat \eta\Big(\frac{\delta\gamma_{2m-3} }{2\pi}\Big)  \Big|^2\Bigg\}
\cr&= \sum_{\substack{\gamma_1  }}\Big|\widehat \eta\Big(\frac{\delta\gamma_1 }{2\pi}\Big)  
\Big|^2
\sum_{\substack{\gamma_3\notin\{ \gamma_1,-\gamma_1\}  }}\Big|\widehat \eta\Big(\frac{\delta\gamma_3 }{2\pi}\Big)  \Big|^2\ldots \Bigg\{ s_2 -2(m-1)\Big|\widehat \eta\Big(\frac{\delta\gamma_{2m-3} }{2\pi}\Big)  \Big|^2\Bigg\}
\cr&\geq \frac{M_{2m-2}'}{\mu_{2(m-1)}}s_2  -2(m-1)s_2^{m-2}s_4.
\end{align*} 
The claimed bound follows by induction on $m$.
\end{proof}

We are ready to prove our main theorem. 
\begin{proof}[Proof of Theorem~\ref{theorem main lower bounds moments}]
We begin by applying Lemma~\ref{lemma explicit formula}. Under RH, we set $T:=\log X$ and obtain that
\begin{align*}
&(-1)^n   M_n(\e^T,\delta;\eta,\Phi)=\frac {(-1)^n}{ T\int_0^\infty \Phi}\int_{0}^{\infty} \Phi\Big( \frac {t}{T}\Big)\big(\psi_\eta(\e^t,\delta)-\e^{\frac t2}\delta \L_\eta( \tfrac \delta 2 )\big)^{n} \d t   \\
&=  \frac{\delta^n}{\int_0^\infty \Phi} \sum_{\gamma_1,\dots,\gamma_n} \widehat \eta\Big(\frac{\delta\gamma_1}{2\pi}\Big)\cdots \widehat\eta\Big(\frac{\delta\gamma_n}{2\pi}\Big)\int_0^{\infty} \e^{i tT (\gamma_1+\dots+\gamma_n) }\Phi(t) \d t +O\Big(\frac{ \delta (K_{\eta} \log(\delta^{-1}+2))^n}T \Big)\\
&= \frac{ \delta^n}{2\int_0^\infty \Phi} \sum_{\gamma_1,\dots,\gamma_n} \widehat \Phi\Big(\frac {T(\gamma_1+\dots+\gamma_n)}{2\pi}\Big)\widehat \eta\Big(\frac{\delta\gamma_1}{2\pi}\Big)\cdots \widehat\eta\Big(\frac{\delta\gamma_n}{2\pi}\Big) +O\Big(\frac{ \delta (K_{\eta} \log(\delta^{-1}+2))^n}T \Big),
\end{align*}
since both $\Phi$ and $\widehat \Phi$ are even and real-valued.
Here, $\gamma_1,\dots, \gamma_n$ are running over the imaginary parts of the non-trivial zeros of $\zeta(s)$.
If $n$ is odd, then the claimed estimate follows from discarding the sum over zeros entirely. If $n$ is even, then by positivity of $ \widehat\eta$ and $\widehat \Phi$ we may only keep the terms for which $\gamma_1+\dots+\gamma_n=0$, and apply Lemma~\ref{lemma lower bound 2k sum over zeros}. The claimed lower bound follows.
\end{proof}

\section{Proof of Corollaries~\ref{corollary} and~\ref{corollary very large values}}

We first need to establish the following proposition, which is strongly inspired from the work of Kaczorowski and Pintz~\cite{KP86}.
We consider
$$ F(x,\delta;\eta):=- \delta\sum_{\rho} \frac{x^{  \rho-\frac 12 }}{\rho-\frac 12} \widehat \eta\Big(\frac{\delta}{2\pi}\frac{\rho-\frac 12}i\Big),$$
which is readily shown to be real-valued by grouping conjugate zeros.

\begin{proposition}
\label{proposition omega result}
Assume that RH is false, and let $\eta\in \E_\kappa$ with $0<\kappa < \frac 12$. Then, there exists an absolute (ineffective) constant $\theta>0$ and a sequence $\{x_j\}_{j\geq 1}$ tending to infinity such that for each $j\geq 1$ and
uniformly 
for $ x_j^{-\theta }\leq  \delta \leq \delta_{\eta}$, where $\delta_{\eta}>0$ is small enough, we have that
  $$F(x_j,\delta;\eta)
  > x_j^{\theta }.  $$ 
\end{proposition}
  \begin{proof}
  Consider, for $\Theta>0$, the $(n-1)$-fold average 
    \begin{align*}
    F_n(\e^t,\delta,\Theta;\eta)&:=
    - \delta\sum_{\rho} \frac{\e^{  ( \rho-\frac 12)t }}{(\rho-\frac 12)^n} \widehat \eta\Big(\frac{\delta}{2\pi}\frac{\rho-\frac 12}i\Big) -\delta \frac{\e^{\Theta t}}{\Theta^{n-1}},
    \end{align*} 
so that $ \frac{\d^{n-1}}{(\d t)^{n-1}}F_n(\e^t,\delta,\Theta;\eta)= F(\e^t,\delta;\eta)-\delta \e^{\Theta t} $.
Let $\rho_e=\beta_e+i\gamma_e$ be a zero of $\zeta(s)$ violating RH, of least positive imaginary part $\gamma_e$, and such that there is no other zero of imaginary part equal to $\gamma_e$ but of greater real part. Let moreover $\eps < \beta_e-\frac 12$. We will show that $F_n(\e^t,\delta,\Theta;\eta)=0$ for many values of $t$ (independently of $\delta$), and then apply Rolle's theorem. 

We pick $t=cn$, with $n\geq 1$ and $c\in \mathbb R$.
If $\Theta \leq  \eps$ and $c$ is large enough in terms of $\eps$ and $\Theta$, say $c\geq c_0(\eps)$ (later we will require that $c_0(\eps)\geq 1$), then
$$ \frac{\e^{cn\Theta}}{\Theta^{n-1}} <\Big(\frac{\e^{c(\beta_e-\frac 12)}}{2|\rho_e-\frac 12|}\Big)^n.
$$
We will also impose $c$ to be bounded in terms of $\eps$ and $\rho_e$, say $c\leq c_1(\eps)$. More precisely, we pick $c_1(\eps)= c_0(\eps) + 2$.
 Then, there exists $U_\eps$ large enough so that 
$$ \sum_{|\Im m(\rho)|>U_\eps} \frac{\e^{ c n (\rho-\frac 12) }}{(\rho-\frac 12)^n} \widehat \eta\Big(\frac{\delta}{2\pi}\frac{\rho-\frac 12}i\Big) \ll_{\kappa,\eta} (\log U_\eps) \frac{\e^{c\frac n2}}{  U_\eps^{n-1} }< \Big(\frac{\e^{c(\beta_e-\frac 12)}}{2|\rho_e-\frac 12|}\Big)^n, $$
whenever $\delta \leq  \kappa$, $n>n_0(\eps)$ and $c_0(\eps) < c < c_1(\eps)$. Here we used the bound 
$$  \widehat \eta(s) = \int_{\mathbb R} \e^{ -2\pi i s x} \eta(x)\d x \ll \int_{0}^{\infty} \e^{ 2\pi  |\Im m(s)| x} \e^{-\kappa x}\d x\ll \frac 1{ \kappa - 2\pi |\Im m(s)|}\qquad (|\Im m(s)| <  \kappa / 2\pi) .  $$
We conclude that under these last two conditions, 
$$ F_n(\e^{cn},\delta ,\Theta;\eta) = -\delta \sum_{|\Im m(\rho)|\leq U_\eps} \frac{\e^{ c n (\rho-\frac 12) }}{(\rho-\frac 12)^n} \widehat \eta\Big(\frac{\delta}{2\pi}\frac{\rho-\frac 12}i\Big) +O\Big(\delta\Big(\frac{\e^{c(\beta_e-\frac 12)}}{2|\rho_e-\frac 12|}\Big)^n\Big).  $$

For two distinct zeros $\rho_1,\rho_2$ of $\zeta(s)$ of positive imaginary part at most $U_\eps$, consider the function
 \begin{align*} f: (c_0(\eps),c_1(\eps))& \rightarrow \mathbb R\\
   c & \mapsto c(\Re e(\rho_1)-\tfrac 12)-\log|\rho_1-\tfrac 12|-c(\Re e(\rho_2)-\tfrac 12)+\log|\rho_2-\tfrac 12|. 
\end{align*}
This linear function is not identically zero and has at most one zero, hence there exists a subset $S_1 \subset (c_0(\eps),c_1(\eps))$ which is a union of two intervals such that for all $c\in S_1$, $|f(c)|\geq \kappa_\eps$, for some fixed and small enough $\kappa_\eps>0$. By picking $\kappa_\eps$ small enough, we may require that $\lambda(S_1) \geq 2-2^{-\#\{ \rho \, :\, \zeta(\rho)=0,\, |\Im m(\rho)| \leq U_\eps \}}$, where $\lambda$ is the Lebesgue measure.
We may iterate this procedure with all pairs of distinct zeros $\rho_j,\rho_k$ such that $0<\Im m(\rho_j),\Im m(\rho_k) \leq U_\eps$, and deduce that there exists a subset $S\subset (c_0(\eps),c_1(\eps))$ of measure $\geq 1$ which is a disjoint union of at most $2^{\#\{ \rho \, :\, \zeta(\rho)=0,\,  0 < \Im m(\rho) \leq U_\eps \}} +1$ intervals $(\alpha_j,\tau_j)$ such that for each $j$ and whenever $c\in(\alpha_j,\tau_j)$, there exists a zero $\rho_j=\beta_j+i\gamma_j$ such that 
$$ c(\Re e(\rho_j)-\tfrac 12)-\log|\rho_j-\tfrac 12| - \max\{ c(\Re e(\rho)-\tfrac 12)-\log|\rho-\tfrac 12| : \zeta(\rho)=0,0<\Im m(\rho)\leq U_\eps\}\geq \kappa_\eps. $$
Then, denoting by $m_j$ the multiplicity of $\rho_j$, for all $c\in (\alpha_j,\tau_j)$ we have that
$$ F_n(\e^{cn},\delta,\Theta;\eta ) = -\delta m_j \Re e\Big(\frac{\e^{ c n (\rho_j-\frac 12) }}{(\rho_j-\frac 12)^n} \widehat \eta\Big(\frac{\delta}{2\pi}\frac{\rho_j-\frac 12}i\Big) \Big)+O\Big(\delta\Big(\frac{K_\eps\e^{c(\beta_j-\frac 12)}}{|\rho_j-\frac 12|}\Big)^n\Big), $$
where $0<K_\eps<1$ is absolute. Note that for all small enough $\delta$ and for all $j$, we have that  $\widehat \eta(\frac{\delta}{2\pi}\frac{\rho_j-\frac 12}i) = \widehat\eta(0)+O(\delta).$  Hence, 
$$ F_n(\e^{cn},\delta,\Theta;\eta ) = -\delta m_j \Re e\Big(\frac{\e^{  cn (\rho_j-\frac 12) }}{(\rho_j-\frac 12)^n}  \Big)\widehat \eta(0)+O\Big(\delta^2 m_j\Big(\frac{\e^{c(\beta_j-\frac 12)}}{|\rho_j-\frac 12|}\Big)^n+\delta\Big(\frac{K_\eps\e^{c(\beta_j-\frac 12)}}{|\rho_j-\frac 12|}\Big)^n\Big). $$ 
For $n$ large enough, this function has at least $ (\tau_j-\alpha_j)\Im m(\rho_j) n/\pi +O(1) \geq 4 (\tau_j-\alpha_j)n $ zeros for $c\in(\alpha_j,\tau_j)$. Indeed, this follows from the intermediate value theorem combined with the identity 
$$ \Re e\Big(\frac{\e^{  cn (\rho_j-\frac 12) }}{(\rho_j-\frac 12)^n}  \Big) = \frac{\e^{ cn(\beta_j -\frac 12)}}{|\rho_j-\frac 12|^n} \cos(\nu_{j,c} n),$$
where $\nu_{j,c}:=\Im m(\rho_j) c-\Im m (\log (\rho_j-\frac 12))$.
Since this is true for every $j$, we conclude that  $F_n(\e^{cn},\delta,\Theta;\eta )$ has at least $4 n \lambda(S) \geq 4n$ zeros for $c\in S$. In other words, $ F_n(\e^t,\delta,\Theta;\eta)$ has at least $4 n $ zeros for $t\in [c_0(\eps) n , c_1(\eps) n ]$. By Rolle's theorem, we deduce that $F(\e^t,\delta;\eta) -\delta \e^{\Theta t}$  has at least $3 n$ zeros on this interval (note that by our conditions on $\eta$, $F(\e^t,\delta;\eta)$ is continuous). In the range $\e^{-\theta t}\leq  \delta $, the result follows whenever $0<\theta<\Theta/2$.
  \end{proof}

  We are ready to prove our first unconditional result.

\begin{proof}[Proof of Corollary~\ref{corollary}]
 If RH is true, then this is a particular case of Theorem~\ref{theorem main lower bounds moments}. Let us then assume that RH is false. 
By H\"older's inequality we have that 
\begin{align*}
  M_{2m}(X,\delta;\eta,\Phi) ^{\frac 1{2m}} &\geq \frac  1 {(\log X)\int_0^\infty \Phi} \int_1^{\infty}  \Phi \Big( \frac{\log x}{\log X} \Big)\big|\psi_{\eta}(x,\delta)-x^{\frac 12  } \delta \L_\eta(\tfrac \delta2)\big| \frac{\d x}x \\
& \geq \frac  {c(\Phi)} {(\log X)\int_0^\infty \Phi} \int_1^{ X^{\kappa(\Phi)}} \big(\psi_{\eta}(x,\delta)-x^{\frac 12  } \delta \L_\eta(\tfrac \delta2)\big) \frac{\d x}x,
\end{align*}
where $c(\Phi),\kappa(\Phi)>0$.
By Lemma~\ref{lemma explicit formula}, the integral is equal to
$$- \delta\sum_{\rho} \frac{X^{\kappa(\Phi)    (\rho-\frac 12) }}{\rho-\frac 12} \widehat \eta\Big(\frac{\delta}{2\pi}\frac{\rho-\frac 12}i\Big)+O_{\Phi,\eta}\big( {\delta}(\log (\delta^{-1}+2))^2\big),
  $$ 
  by the Riemann-von Mangoldt formula~\eqref{equation Riemann von Mangoldt}.
The claimed $\Omega$-result then follows from Proposition~\ref{proposition omega result}.
 \end{proof}

In order to prove Corollary~\ref{corollary very large values}, we will apply Theorem~\ref{theorem main lower bounds moments} with $\eta(u)= \max(0,1-|u|)$. This is not an element of $\E_\kappa$ since it is not differentiable. However, as remarked in the introduction, one can go through the proof of Lemma~\ref{lemma explicit formula} and check that it applies when $\eta$ is Lipschitz, compactly supported, and monotonous on $\mathbb R_{\geq 0}$; we deduce that the same is true for Theorem~\ref{theorem main lower bounds moments} (note that the conditions of Lemma~\ref{lemma convergent sum} are satisfied for $h=\widehat \eta^2$). 
 
\begin{proof}[Proof of Corollary~\ref{corollary very large values}]
If RH is false, then the result follows from an adaptation of the proof of Proposition~\ref{proposition omega result}. Rather than going through the proof, we highlight the two major differences. Firstly, the function we need to study is
$$ - \sum_{\rho} \frac{\e^{  \rho t }}{\rho^n}  ((1+\delta)^\rho-1) -\delta \frac{\e^{(\frac 12+\Theta) t}}{(\frac 12+\Theta)^{n-1}},$$
which has the weight $(1+\delta)^\rho-1$ instead of $\delta\widehat \eta(\frac{\delta}{2\pi}\frac{\rho-\frac 12}i)$. However, this weight is $\ll \delta |\rho|$ uniformly for all $0<\delta \leq 1$ and  $0<\Re e(\rho)<1$.
The second major difference is 
the proof that the existence of two zeros of the continuous and piecewise differentiable function $$ - \sum_{\rho} \frac{\e^{    \rho t }}{\rho^2}   \big((1+\delta)^\rho-1\big) -\delta \frac{\e^{(\frac 12+\Theta )t}}{\frac 12+\Theta}$$
implies that the piecewise continuous function $$ - \sum_{\rho} \frac{\e^{  \rho t }}{\rho}  \big((1+\delta)^\rho-1\big) -\delta \e^{(\frac 12+\Theta) t}$$ has at least one non-negative value between those zeros. This can be done using a straightforward generalization of Rolle's theorem, which states that if $f$ is continuous on $[a,b]$ for which $f(a)=f(b)$ and the one-sided derivatives 
$$ f^{\pm }(c):=\lim_{x\rightarrow c^{\pm}} \frac{f(x)-f(c)}{x-c} $$
exist for all $c\in (a,b)$, then there exists $c_0\in (a,b) $ such that $f^+(c_0)f^-(c_0)\leq 0$. The rest of the proof is similar.

We now assume RH.
Let us also assume that for all large enough $x$ and for all $\delta'$ in the range $\frac{\eps_0(\log_3 x)^{\frac 92}}{4(\log x)^{2} (\log_2 x)^{\frac 52}} \leq \delta' \leq 2\frac{(\log_3 x)^3}{(\log_2 x)^{2}}$ we have that
$$ \big|\psi(x+\delta' x)-\psi(x) - \delta' x \big| \leq \eps_0   \delta'^{-\frac 14}(\log(\delta'^{-1}+2))^{\frac 14} \cdot \big(\delta' x \log(\delta'^{-1}+2)\big)^{\frac 12}, $$
where $\eps_0>0$ is the implied constant in the first error term in~\eqref{equation nth moment first bound}.

Define $\eta(u):= \max(0,1-|u|) $, which is even, non-negative, compactly supported and monotonous for $u\geq 0$. Moreover, $\widehat \eta (\xi) = (\sin(\pi \xi)/(\pi \xi))^2\geq 0$.
Now, for any $0<\delta \leq 1$, $x\geq 1$ and $ x\e^{-\delta}\leq n\leq  x\e^\delta$, we write $\eta(\delta^{-1} \log(\frac nx)) = 1-\delta^{-1}|\int_n^x \frac{\d t}t|$ and deduce that
\begin{align*}
& \sum_{n\geq 1} \Lambda(n) \eta \Big(\delta^{-1} \log \Big( \frac nx\Big) \Big) -x \delta \mathcal L_\eta(\delta )= \psi(x\e^\delta)-\psi(x\e^{-\delta})\\ &\quad - \delta^{-1}\Big( \int_x^{x\e^\delta} \Big( \sum_{ t < n\leq x\e^\delta} \Lambda(n) \Big) \frac{\d t}t	 + \int_{x\e^{-\delta}}^x \Big( \sum_{  x\e^{-\delta}< n\leq t } \Lambda(n) \Big) \frac{\d t}t	\Big)- x\delta \int_{\mathbb R} \eta(u)\e^{\delta u}\d u 
\\
 & = \psi(x\e^\delta)-\psi(x\e^{-\delta})-2x\sinh(\delta) \\ &\quad-\delta^{-1}\Big( \int_x^{x \e^A} \big(\psi(x\e^\delta)-\psi(t)-(x\e^\delta-t )\big) \frac{\d t}t	  + \int_{x\e^{-A}}^x  \big(\psi(t)-\psi( x\e^{-\delta})-(t-x\e^{-\delta} )\big) \frac{\d t}t	\Big)\\ &\hspace{3cm}
 +O\big(\delta^{-1} x^{\frac 12}(\log x)^2 (\delta-A) \big),
\end{align*}
for any $0<A<\delta$; in particular for $A=\delta - \eps_0\delta^{\frac 54}(\log(\delta^{-1}+2))^{\frac 34}/(\log x)^2$.
Here we used the (trivial) RH bound
$$\psi(M)-\psi(N)-(M-N) \ll  M^{\frac 12} (\log (M+2))^2 \hspace{2cm} ( 1\leq N\leq M). $$

By our hypothesis, we deduce that for $X$ large enough, $m\geq 2$, $\delta = (\log_3 X )^3/(\log_2 X)^2$ and in the range $ \exp((\log X)^{\frac 12}) \leq x \leq X$,
$$ x^{-\frac 12} \Big( \sum_{n\geq 1} \Lambda(n) \eta \Big(\delta^{-1} \log \Big( \frac nx\Big) \Big)  - x \delta  \mathcal L_\eta(\delta )\Big) \ll \eps_0 \delta^{\frac 14}  \big(\log(\delta^{-1}+2)\big)^{\frac 34}. $$
Combining this with~\eqref{equation explicit formula psi approximation} with $\eps=0$ (since $\widehat \eta (s) \ll (1+|s|)^{-2} $ for $|\Im m(s)|\leq \frac 12$, and recalling that the differentiability condition in Lemma~\ref{lemma explicit formula} can be replaced by one of Lipschitz since $\eta$ has compact support and is decreasing on $\mathbb R_{\geq 0}$), we deduce that 
$$ \psi_\eta(x,\delta) - x^{\frac 12} \delta  \mathcal L_\eta(\tfrac \delta 2 )\ll \eps_0 \delta^{\frac 14}  \big(\log(\delta^{-1}+2)\big)^{\frac 34}. $$
Now, making the choice
$\Phi=\eta$, this implies that for $X$ large enough,
\begin{align*}
 M_n(X,\delta;\eta,\Phi) &=\frac {1}{ (\log X) \int_{0}^{\infty }\Phi}\int_{\exp((\log X)^{\frac 12})}^{\infty} \Phi\Big( \frac {\log x}{\log X}\Big)\big(\psi_\eta (x,\delta;\eta)-x^{\frac 12} \delta \L_\eta( \tfrac \delta 2 )\big)^{2m} \frac{\d x}x \\ &\hspace{4cm} +O\big((\log X)^{-\frac 12} (K^{\frac 12}\eps_0 \log(\delta^{-1}+2))^{2m}\big)\\ 
& \ll \big(K \eps_0^2\delta^{\frac 12}  (\log(\delta^{-1}+2))^{\frac 32}\big)^m+(\log X)^{-\frac 12} \big(K^{\frac 12}\eps_0 \log(\delta^{-1}+2)\big)^{2m},
\end{align*}
where $K>0$ is absolute and 
where we have bounded the part of the integral with $x\leq \exp((\log X)^{\frac 12})$ using the uniform bound in Lemma~\ref{lemma explicit formula}. Recalling that $\delta = (\log_3 X )^3/(\log_2 X)^2$, 
for $ \eps_0^2 \delta^{-\frac 12} (\log(\delta^{-1}+2))^{\frac 12}  \leq m \leq \eps_0 \delta^{-\frac 12} (\log(\delta^{-1}+2))^{\frac 12} $, we have that $(\log X)^{- \frac 1{m-1}} \leq \delta  \leq \delta_0$, and hence 
Theorem~\ref{theorem main lower bounds moments} implies the lower bound
\begin{align*} M_n(X,\delta;\eta,\Phi)&  \geq (1+O(\eps_0^2)) \mu_{2m} \big( \tfrac 23\delta  \log(\delta^{-1}+2)\big)^m \cr& \geq (2\pi)^{\frac 12} (1+O(\eps_0^2)) \big( \tfrac {2m}{3\e}\delta  \log(\delta^{-1}+2)\big)^m. 
\end{align*}
When $\eps_0$ is small enough, we obtain a contradiction as soon as the range $$ \eps_0^2 K (\tfrac {3}{2} \e+\eps_0) \delta^{-\frac 12} (\log (\delta^{-1}+2))^{\frac 12} \leq m\leq  \eps_0   \delta^{-\frac 12} (\log (\delta^{-1}+2))^{\frac 12}$$ contains an integer; this is clearly the case when $\eps_0$ is small enough and $X$ is large enough. The proof of the first statement follows. The proof of the second is similar. 
\end{proof}

\end{document}